\documentclass[10pt]{article}

\usepackage{epsf,amsmath,amsfonts,amssymb,amsthm,graphicx,mathrsfs,latexsym,enumerate,leftidx,subfigure}
\usepackage{epsfig} 
\usepackage[text={7.0in,9.5in},centering]{geometry} 
\usepackage{multirow}

\newtheorem{cl}{Claim}[section]
\newtheorem{prop}{Proposition}[section]

\newtheorem{cor}{Corollary}[section]
\newtheorem{thm}{Theorem}[section]
\newtheorem{no}{Note}[section]
\newtheorem{re}{Remark}[section]

\newtheorem*{acknowledgement*}{Acknowledgement}

\makeatletter
\newcommand{\vast}{\bBigg@{4}}
\newcommand{\Vast}{\bBigg@{5}}
\makeatother

\begin{document}
\title{Modeling Stochastic Anomalies in an SIS and SIRS System}
\author{Andrew Vlasic \\ Department of Mathematics \\ University of Illinois at Urbana-Champaign}
\date{}

\maketitle

\begin{abstract}
I propose a stochastic SIS and SIRS system to include a Poisson measure term to model anomalies in the dynamics. In particular, the positive integrand in the Poisson term is intended to model quarantine. Conditions are given for the stability of the disease free equilibrium for both systems. 
\end{abstract}


\section{Introduction}
Many authors have considered continuous time stochastic epidemiological models \cite{RWZ12,TBV05,YJOA12}, however, none of these models taken into account anomalies that affect the system. Examples of these events are super-carriers, the amount of contacts of individuals in the stadium and in the bars during a sports event, quarantine, etc. A particular example is the three waves of influenza pandemic in 1918 \cite{H11}.

\medskip
For a fixed population of size $N$, we consider the following SIS model:
\begin{equation}\begin{split}
S'(t) &  = -\beta S(t) I(t) -\mu S(t) + \mu + \lambda I(t) ,
\\ I'(t) & = \beta S(t) I(t) - (\lambda +\mu) I(t),
\end{split}\end{equation}
and the SIRS model:  
\begin{equation}\begin{split}
S'(t) & = -\beta S(t) I(t) +\delta R(t), 
\\ I'(t) & =\beta S(t) I(t) - \lambda I(t), 
\\ R'(t) & = \lambda  I(t)  - \delta R(t).
\end{split}\end{equation}
Here, $S(t)$, $I(t)$, and $R(t)$, denote the frequencies of the susceptible, infected, and removed respectively, where $S(t)+ I(t) +R(t)=1$. The constant $\mu$ represents the birth and death rate (newborns are assumed to be susceptible), $\lambda$ is the recovery rate for the individuals that are infected, $\beta$ represents the average number of contacts per day, and $\delta$ is the rate for which recovered individuals become susceptible. 

\medskip
For the SIS model, if $\beta > \mu + \lambda$, then an epidemic will occur, and if $\beta \leq \mu + \lambda$ then the process will converge to the disease-free equilibrium, which means that the disease has disappeared. In other words, if $\beta \leq \mu + \lambda$ then point $(1,0)$ is globally asymptotically stable, and if $\beta > \mu + \lambda$, the point $\displaystyle \Big( \frac{ \mu + \lambda}{ \beta }, 1- \frac{ \mu + \lambda}{ \beta } \Big)$ (the endemic equilibrium) is globally asymptotically stable.

\medskip
Considering the SIRS model, when $\beta \leq \lambda$ the disease free equilibrium is globally asymptotically stable. If $\beta > \lambda$ then the endemic equilibrium $\displaystyle \bigg( \frac{ \lambda}{ \beta }, 1- \frac{ \lambda}{ \beta } - \frac{ 1-\frac{\lambda}{\beta}  }{ 1 + \frac{\delta}{\lambda}  } , \frac{ 1-\frac{\lambda}{\beta}  }{ 1 + \frac{\delta}{\lambda}  } \bigg)$ is globally asymptotically stable.

\medskip
We now define and give some intuition to the Poisson measure. Take $X$ to be an $\mathbb{R}$-valued L\'evy process on $(\Omega, \mathcal{F} , P)$ with the filtration $\{ \mathcal{F}_t \}_{ t\in\mathbb{R}_+ }$, where $\mathcal{F}_0$ contains all of the null sets of $\mathcal{F}$. For $B \in \mathcal{B}\Big( \mathbb{R}^d \backslash \{0\} \Big)$ and a fixed $\omega \in \Omega$, define $\Delta X_{\omega}(s) := X_{\omega}(s) - X_{\omega}(s-)$ and $\displaystyle N_{\omega}(t,B) := \# \Big\{0 \leq s \leq t: \Delta X_{\omega}(s) \in B \Big\}$. This random counting measure is known as the \textbf{Poisson measure} since, fixing $B$, the map $t\to N_{\omega}(t,B)$ is a Poisson process with intensity $t\nu(B):=E\Big[N_{\omega}(t,B)\Big]$. We call $\nu(\cdot)$ the \textbf{intensity measure} and it is well known that it is a L\'evy measure, hence, it is a Borel Measure with $\displaystyle \int_{ \mathbb{R}\backslash\{0\} } \Big( |x|^2 \wedge 1 \Big)\nu(dx)<\infty$. The measure $\tilde{N}_{\omega}(dt,dx):=N_{\omega}(dt,dx)-dt\nu(dx)$ is called the \textbf{compensated Poisson measure} \Big(where $dt$ is the Lebesgue measure\Big). For $B \in \mathcal{B}\Big( \mathbb{R}^d \backslash \{0\} \Big)$ where $0\not\in \overline{B}$, one can show $\tilde{N}_{\omega}(t,B)$ is a martingale and $E \Big[ \tilde{N}_{\omega}(t, B ) \Big]=0$, (which stems from $\nu(B)<\infty$).  See Sato \cite{S99}, Bertoin \cite{B96}, or Applebaum \cite{DA04} for further information.

\medskip
We are ready to stochastically perturb the systems above. Define $W(t)$ as a standard Brownian motion and $\sigma^2$ as the variance of the contacts. Furthermore, take $N(dt,dy)$ as a Poisson measure independent of $W(t)$ and $\nu(\cdot)$ as the intensity measure. We assume that $\nu$ is a L\'evy measure and $\nu(\mathbb{R})<\infty$. Lastly, take $h(y)$ as the affect of random jumps in the population, where $-1<h(y)<1$ for every $y \in \mathbb{R}$, and $h(y)$ is continuously differentiable. For the SIS model, Equation (1) becomes the stochastic differential equation
\begin{equation}\begin{split}
dS(t) & = \bigg(-\beta S(t) I(t) -\mu S(t) + \mu + \lambda I(t)  \bigg)dt -\sigma  S(t) I(t) dW(t)  + \int_{ \mathbb{R} } h(y) S(t-) I(t-) N(dt,dy), 
\\ dI(t) & = \bigg( \beta S(t) I(t) - (\lambda +\mu) I(t)  \bigg)dt + \sigma S(t) I(t) dW(t)  - \int_{ \mathbb{R} } h(y) S(t-) I(t-) N(dt,dy).
\end{split}\end{equation}

\medskip
For the SIRS dynamic, since the anomaly term captures the effects of quarantine, we will take a function similar to $h(y)$, call it $j(y)$, with the same assumptions, except that $j(y)$ is assumed to be nonnegative. Equation (2) is now a stochastic differential equation of the form
\begin{equation}\begin{split}
dS(t) & = \Big( -\beta S(t) I(t) +\delta R(t) \Big)dt - \sigma  S(t) I(t) dW(t)  , 
\\ dI(t) & = \bigg( \beta S(t) I(t) - \lambda I(t)  \bigg)dt + \sigma S(t) I(t) dW(t)  -  \int_{ \mathbb{R} } j(y) I(t-) N(dt,dy), 
\\ dR(t) & = \Big( \lambda  I(t)  - \delta R(t) \Big) dt  +  \int_{ \mathbb{R} } j(y) I(t-) N(dt,dy). 
\end{split}\end{equation}

\medskip
In this paper, we will show that Equation (3) and Equation (4) are well-defined, and we give conditions for stability of the disease free equilibrium for both dynamics.

\medskip
Given that the initial condition is in the interior of the simplex, we will show that for all finite time, Equations (3) and (4) are almost surely in the simplex. We will only show this property for the stochastic SIS, since the stochastic SIRS model may be shown in a similar manner. Define $K(t)=\big( S(t),I(t) \big)$, $\Delta_2 = \Big\{ \mathbf{y} \in \mathbb{R}^2 : x_1,x_2>0 \ \mbox{and} \ x_1 + x_2 =1  \Big\}$, and $[ \cdot, \cdot]$ as quadratic variation. For simplicity, we will define $S_c(t)$ and $I_c(t)$ as the continuous part of the process. 

\bigskip
\bigskip
\begin{prop}
For all finite $t$, given that $\mathbf{x} \in \Delta_2$, $P_{ \mathbf{x} } \Big( K (t)\in \Delta_2 \Big)=1$.
\end{prop}
\begin{proof}
Consider the mapping on the simplex $G( \mathbf{x} )= x_1+ x_2$, and define $U(t)= G\big( L(t) \big)$. It\^o's lemma yields
\begin{equation*}\begin{split}
dU(t) &  = \frac{ \partial G}{ \partial y_1}\big( L(t) \big) dS_c(t) +\frac{1}{2} \frac{ \partial G}{ \partial y_2}\big( L(t) \big) dI_c(t)  + \frac{1}{2}  \frac{ \partial^2 G}{ \partial y_1^2 }\big( L(t) \big) \big[ dS_c(t), dS_c(t) \big]
\\ &  + \frac{1}{2}  \frac{ \partial^2 G}{ \partial y_2^2 }\big( L(t) \big) \big[ dI_c(t), dI_c(t) \big] + \frac{1}{2}  \frac{ \partial^2 G}{ \partial y_1 \partial y_2 }\big( L(t) \big) \big[ dS_c(t), dI_c(t) \big]  + \frac{1}{2}  \frac{ \partial^2 G}{ \partial y_2 \partial y_1 } \big( L(t) \big) \big[ dI_c(t), dS_c(t) \big]
\\ & + \int_{ \mathbb{R} } \bigg[ G\Big( L(t) + \big( h(y) S(t-) I(t-), -  h(y) S(t-) I(t-) \big) \Big) - G\big(  L(t)  \big) \bigg] N(dt,dy)
\\ & = \bigg(-\beta S(t) I(t) -\mu S(t) + \mu + \lambda I(t)  \bigg)dt -\sigma  S(t) I(t) dW(t) + \bigg( \beta S(t) I(t) - (\lambda +\mu) I(t)  \bigg)dt + \sigma S(t) I(t) dW(t) 
\\ & + \int_{ \mathbb{R} } \bigg[ \Big( S(t) + h(y) S(t-) I(t-) + I(t) -  h(y) S(t-) I(t-) \Big) - \Big(  S(t) + I(t) \Big) \bigg] N(dt,dy)
\\ & = \bigg(\ -\mu \Big( S(t) + I(t) \Big) + \mu  \bigg)dt.
\end{split}\end{equation*}
Thus, if the process is in $\Delta_2$, then $\displaystyle dU(t) = \bigg(\ -\mu \Big( S(t) + I(t) \Big) + \mu  \bigg)dt = \big( -\mu + \mu)dt =0$. Therefore, if $\displaystyle L(t) \in \Delta_2$  then $U(t)=1$.

\medskip
Finally, we show that $L(t)$ does not hit or jump over the boundary in finite time. Define $\Psi(\mathbf{x}) = \log( x_2 / x_1 )$, $\tau$ as the first time $L(t)$ leaves the open simplex (i.e., such that $I( \tau ) \leq 0$  or $S( \tau ) \leq 0$ ), and $Z(t) := \Psi( L(t) )$ for $t<\tau$. We will apply Theorem 2.1 in Meyn and Tweedie \cite{MT93} to the process $Z(t)$ in order to show this $Z(t)$ does not explode in finite time, and thus $P_{ \mathbf{x} } \big( \tau = \infty \big)=1$. By It\^o's lemma we have  
 \begin{equation*}\begin{split}
 dZ(t) &  =  \frac{ \partial \Psi }{ \partial x_1}\big( L(t) \big) dS_c(t) + \frac{ \partial \Psi }{ \partial x_2}\big( L(t) \big) dI_c(t)  + \frac{1}{2}  \frac{ \partial^2 \Psi }{ \partial x_1^2 }\big( R(t) \big) \big[ dS_c(t), dS_c(t) \big]
\\  + & \frac{1}{2}  \frac{ \partial^2 \Psi }{ \partial x_2^2 }\big( L(t) \big) \big[ dI_c(t), dI_c(t) \big] + \frac{1}{2} \frac{ \partial^2 \Psi }{ \partial x_1 \partial x_2 }\big( L(t) \big) \big[ dS_c(t), dI_c(t) \big] +  \frac{1}{2}  \frac{ \partial^2 \Psi }{ \partial x_2 \partial x_1 } \big( L(t) \big) \big[ dI_c(t), dS_c(t) \big]
\\  + & \int_{ \mathbb{R} } \bigg[ \Psi \Big( L(t) + \big( h(y) S(t-) I(t-), -  h(y) S(t-) I(t-) \big) \Big) -  \Psi \big(  L(t)  \big) \bigg] N(dt,dy)
\end{split}\end{equation*}
 \begin{equation*}\begin{split}
\\  \hspace{50pt} & =  \frac{-1}{ S(t) }\bigg( \bigg(-\beta S(t) I(t) -\mu S(t) + \mu + \lambda I(t)  \bigg)dt -\sigma  S(t) I(t) dW(t) \bigg)
\\ & + \frac{1}{ I(t) }\bigg(  \bigg( \beta S(t) I(t) - (\lambda +\mu) I(t)  \bigg)dt + \sigma S(t) I(t) dW(t)  \bigg)
\\ & +  \frac{ \sigma^2  }{ 2 S^2(t) } S^2(t) I^2(t) dt +  \frac{ -\sigma^2 }{ 2 I^2(t) } S^2(t) I^2(t) dt
\\ & + \int_{ \mathbb{R} } \bigg[ \log\bigg( \frac{ I(t) -  h(y) S(t-) I(t-) }{ S(t) + h(y) S(t-) I(t-)} \bigg) -  \log\big( I(t)/S(t)  \big) \bigg] N(dt,dy)
\\ & = \bigg( \beta I(t) + \mu + \frac{-\mu}{ S(t) } - \lambda \frac{ I(t) }{ S(t) }  \bigg)dt + \sigma I(t) dW(t) +  \bigg( \beta S(t) - (\lambda +\mu) \bigg)dt + \sigma S(t) dW(t) 
\\ & +  \frac{ \sigma^2  }{ 2 } I^2(t) dt +  \frac{ -\sigma^2 }{ 2  } S^2(t) dt + \int_{ \mathbb{R} } \log\bigg( \frac{ 1 -  h(y) S(t-) }{ 1 + h(y) I(t-)} \bigg)N(dt,dy)
\\ & = \bigg( \beta \Psi^{-1}_1\big( Z(t) \big) + \mu + \frac{-\mu}{ \Psi^{-1}_2 \big( Z(t) \big) } - \lambda e^{ Z(t) } + \beta \Psi^{-1}_2 \big( Z(t) \big) - (\lambda +\mu)  + \frac{ \sigma^2 }{ 2 } \Psi^{-1}_1\big( Z(t) \big)^2 +  \frac{ -\sigma^2 }{ 2  } \Psi^{-1}_2\big( Z(t) \big)^2  \bigg)dt 
\\ & + \sigma  \bigg( \Psi^{-1}_1\big( Z(t) \big) + \Psi^{-1}_2\big( Z(t) \big)  \bigg) dW(t)   + \int_{ \mathbb{R} } \log\bigg( \frac{ 1 -  h(y) \Psi^{-1}_2\big( Z(t) \big) }{ 1 + h(y) \Psi^{-1}_1\big( Z(t) \big) } \bigg)N(dt,dy),
\end{split}\end{equation*}
where $\displaystyle \Psi^{-1}\big( x \big) = \frac{1}{ 1+e^x }\big(  1, e^x \big) = \Big( \Psi_1^{-1}\big( y \big) , \Psi_2^{-1}\big( y \big)  \Big)$.

\medskip Now, defining $\frak{B}$ as the infinitesimal generator for $Z(t)$ and $V(x)=1+x^2$, we see that  
\begin{equation*}\begin{split}
\frak{B} V(x) & = \bigg( \beta \Psi^{-1}_1\big( x \big) + \mu + \frac{-\mu}{ \Psi^{-1}_2 \big( x \big) } - \lambda e^{ x } + \beta \Psi^{-1}_2 \big( x \big) - (\lambda +\mu)  + \frac{ \sigma^2 }{ 2 } \Psi^{-1}_1\big( x \big)^2 +  \frac{ -\sigma^2 }{ 2  } \Psi^{-1}_2\big( x \big)^2  \bigg) 2 x
\\& + \sigma^2  \bigg( \Psi^{-1}_1\big( x \big) + \Psi^{-1}_2\big( x \big)  \bigg)^2  + \int_{ \mathbb{R} } \Bigg[ \bigg( x + \log\bigg( \frac{ 1 -  h(y) \Psi^{-1}_2\big( x \big) }{ 1 + h(y) \Psi^{-1}_1\big( x \big) } \bigg) \bigg)^2 - x^2 \Bigg] \nu(dy)
\\ & = \Bigg[ \sigma^2  \bigg( \Psi^{-1}_1\big( x \big) + \Psi^{-1}_2\big( x \big)  \bigg)^2 + \int_{ \mathbb{R} } \log\bigg( \frac{ 1 -  h(y) \Psi^{-1}_2\big( x \big) }{ 1 + h(y) \Psi^{-1}_1\big( x \big) } \bigg)^2  \nu(dy) \Bigg]
\\ &  + \bigg(  2\beta \Psi^{-1}_1\big( x \big) + 2\mu + \frac{-2\mu}{ \Psi^{-1}_2 \big( x \big) } - 2\lambda e^{ x } + 2\beta \Psi^{-1}_2 \big( x \big) - 2(\lambda +\mu)  + \sigma^2 \Psi^{-1}_1\big( x \big)^2 
\\ & \ \ \ \ \ \   -\sigma^2  \Psi^{-1}_2\big( x \big)^2  + 2 \int_{ \mathbb{R} } \log\bigg( \frac{ 1 -  h(y) \Psi^{-1}_2\big( x \big) }{ 1 + h(y) \Psi^{-1}_1\big( x \big) } \bigg)  \nu(dy)     \bigg) x.
\end{split}\end{equation*}
Noting that $\big| \Psi^{-1}_1\big( x \big) \big| \leq 1 $, $\big| \Psi^{-1}_2\big( x \big) \big| \leq 1$, and $\displaystyle \log\bigg( \frac{ 1 -  h(y) \Psi^{-1}_2\big( x \big) }{ 1 + h(y) \Psi^{-1}_1\big( x \big) } \bigg) \leq \log\bigg( \frac{ 1 -  \min h(y) }{ 1 + \min h(y) } \bigg)$, one can see that there exists positive constants $K$ such that $\displaystyle \frak{B} V(x) \leq K V(x)$. Therefore $P_{ \mathbf{x} } \big( \tau = \infty \big)=1$.

\end{proof}

%SIS
\bigskip
\bigskip
\section{Analysis of the Approximated SIS Model}
Since $I(t)=1-S(t)$ we are able to just focus on $S(t)$, which we are able to rewrite as
\begin{equation}\begin{split}
dS(t) & = \bigg(-\beta S(t) \big( 1-S(t) \big)  -\mu S(t) + \mu + \lambda \big( 1-S(t) \big)  \bigg)dt -\sigma S(t) \big( 1- S(t) \big) dW(t) 
\\ & + \int_{A} h(y) S(t-)\big( 1-S(t-) \big) N(dt,dy) 
\\ & := \alpha( S(t) ) dt + \gamma( S(t) ) d W(t) + \int_{ \mathbb{R} } h(y) S(t-)\big( 1-S(t-) \big) N(dt,dy).
\end{split}\end{equation}
Define  $\displaystyle \tau_{\epsilon}=\inf \Big\{ t \geq 0 : S(t) \geq 1 - \epsilon \Big\}$.

\bigskip
\bigskip
\begin{cl}
For all $0<x<1$, we have $E_x \Big[ \tau_{\epsilon} \Big]< \infty$.
\end{cl}
\begin{proof}
We will follow the proof of Theorem 4.2 given in Imhof \cite{I05}. Take $L$ as the infinitesimal generator for our process $S(t)$ and define $f(y) = e^{\gamma} - e^{ \gamma x}$, where $ \gamma > 0$. Now fix an arbitrarily small $\epsilon>0$. By Dynkin's formula we have that
$$
E_{x_0} \Big[ S\Big(\tau_{\epsilon} \wedge T \Big) \Big] = e^{\gamma} - E_{x_0} \Bigg[ \int _{0}^{ \tau_{\epsilon} \wedge T }Lf \Big( S(t) \Big)dt \Bigg]. 
$$

\medskip
We will now determine an appropriate upper bound  for $Lf \Big( S(t) \Big)$. To adjust for the possibility of the process jumping out of the interval, we will consider a $x\in [0, 1-\epsilon^2]$. Hence
\begin{equation*}\begin{split}
Lf(x) & = -\gamma \Big( -\beta x(1-x) -\mu x + \mu + \lambda \big( 1-x \big) \Big)e^{ \gamma x} - \frac{ \gamma^2 \sigma^2 x^2 \big( 1-x \big)^2}{2}e^{ \gamma x}
\\ & + \int_{ \mathbb{R} } \bigg[ e^{ \gamma x} - e^{ \gamma  \big( x + h(y) x \big( 1-x \big) \big)}  \bigg]\nu(dy)
\\ & = \gamma \bigg( \beta x -\mu  - \lambda  - \frac{\gamma \sigma^2 x^2 \big( 1-x \big) }{2}   \bigg) \big( 1-x \big) e^{ \gamma x}+ \int_{ \mathbb{R} } \bigg[ e^{ \gamma x} - e^{\gamma  \big( x + h(y) x \big( 1-x \big) \big)}  \bigg]\nu(dy)
\end{split}\end{equation*}
Noticing that $ x + h(y) x \big( 1-x \big)$ is positive and recalling the inequality $-e^x \leq -1- x$ for $x>0$, we find that
\begin{equation*}\begin{split}
\int_{ \mathbb{R} } & \bigg[ e^{ \gamma x} - e^{ \gamma  \big( x + h(y) x \big( 1-x \big) \big)}  \bigg]\nu(dy)  \leq \int_{ \mathbb{R} } \bigg[ e^{ \gamma x} - 1 -\gamma  \Big( x + h(y) x \big( 1-x \big) \Big)  \bigg]\nu(dy) 
\\ & \leq \int_{ \mathbb{R} } \bigg[ \sum_{i=2}^{\infty} \frac{\big( \gamma x\big)^i}{i!} -\gamma h(y) x \big( 1-x \big)  \bigg]\nu(dy)   = \gamma  e^{ \gamma x} \int_{ \mathbb{R} } \bigg[ x \sum_{i=2}^{\infty} \frac{\big( \gamma x \big)^{i-1} }{ e^{ \gamma x} i! } - \frac{ h(y) x \big( 1-x \big)  }{ e^{ \gamma x} }  \bigg]\nu(dy) 
\\ & \leq \gamma \big( 1-x \big) e^{\gamma x} \int_{ \mathbb{R} } \bigg[  \frac{1}{ \big( 1-x \big) } - \frac{ h(y) x}{ e^{\gamma x} }  \bigg]\nu(dy).
\end{split}\end{equation*}
Hence
\begin{equation*}\begin{split}
Lf(x) & \leq \gamma \Bigg( \beta x + \frac{ \nu( \mathbb{R} ) }{ \big( 1-x \big) } - \frac{ x}{ e^{ \gamma x} } \int_{ \mathbb{R} } h(y) \nu(dy) - \frac{ \gamma \sigma^2 x^2 \big( 1-x \big) }{2}  - \mu  - \lambda \Bigg) \big(1-x \big) e^{ \gamma x}
\\ & = \gamma \Bigg( \beta x + \frac{ \nu( \mathbb{R} ) }{ \big( 1-x \big) } - \frac{ x }{ e^{ \gamma x} } \int_{ \mathbb{R} } h(y) \nu(dy) - \frac{\gamma \sigma^2 x^2 \big( 1-x \big) }{2} \Bigg) \big( 1-x \big) e^{ \gamma x}  - \gamma \big(  \mu  + \lambda \big) \big( 1-x \big) e^{ \gamma x}.
\end{split}\end{equation*}
Finally, taking $\gamma$ large enough so that $\displaystyle \beta x + \frac{ \nu( \mathbb{R} ) }{ \big( 1-x \big) } - \frac{ x }{ e^{ \gamma x} } \int_{ \mathbb{R} } h(y) \nu(dy) - \frac{\gamma \sigma^2 x^2 \big(1-x \big) }{2}<0$ for all  $x \in [0, 1-\epsilon^2]$ yields
$$
Lf(x) \leq - \gamma \Big(  \mu  + \lambda \Big) \big( 1-x \big)e^{\gamma x}.
$$
Thus
$$
0 \leq E_{y_0} \Big[ S\Big(\tau_{\epsilon} \wedge T \Big) \Big] \leq e^{\gamma} - \epsilon^2 \gamma \Big(  \mu  + \lambda \Big)  E_{y_0} \Big[ \tau_{\epsilon} \wedge T \Big],
$$
and therefore taking $T\to\infty$, the bounded convergence theorem yields $\displaystyle E_{x_0} \Big[ \tau_{\epsilon} \Big]<\infty$.
\end{proof}

\bigskip
\bigskip
\begin{thm}
Suppose that $\displaystyle \int_{ \mathbb{R} } h(y) \nu(dy)<0$. If $\beta < \mu + \lambda + \int_{ \mathbb{R} } h(y) \nu(dy)$, then for $\mathbf{x} \in \Delta_2$,
$$
P_{ \mathbf{x}  } \bigg( \lim_{t \to \infty} K(t) = (1,0) \bigg)=1.
$$
\end{thm}
\begin{proof}
In our proof we will employ the stochastic Lyapunov method. We will focus on $S(t)$, defined by Equation(5). Define $g(x)=1-x$ as our Lyapunov function.  Taking $L_0$ as the infinitesimal generator for $S(t)$, we see that
\begin{equation*}\begin{split}
L_0 g(x) & = \Big(- \beta x \big( 1-x \big) + \mu \big(1-x \big) + \lambda \big( 1-x \big)  \Big)(-1)  + \int_{ \mathbb{R} } \bigg[ \Big\{1 - \Big( x + h(y) x (1-x) \Big) \Big\} - (1-x)  \bigg] \nu(dy)
\\ & = - \bigg( - \beta x + \mu + \lambda +  \int_{ \mathbb{R} } h(y)\nu(dy) x \bigg)\big(1-x \big)
\\ & \leq  - \bigg( - \beta  + \mu + \lambda +  \int_{ \mathbb{R} } h(y)\nu(dy) \bigg)\big( 1-x \big).
\end{split}\end{equation*}
Therefore, Theorem 4 in \cite{GS72} (page 325) tells us that for an $\epsilon>0$, there exists a neighborhood of $(1,0)$, say $U$, such that 
$$
P_{ \mathbf{x}  } \bigg( \lim_{t \to \infty} K(t) = (1,0) \bigg) \geq 1 -\epsilon.
$$
for $\mathbf{x} \in U \cap \Delta_2$.

Now, take an arbitrary $\epsilon>0$ and $\mathbf{x} \in \Delta_2$, and define $M=\Big\{  \lim_{t \to \infty} K(t)  = (1,0) \Big\} $. The strong Markov property tells us that
$$
P_{ \mathbf{x} } \big( M \big) = E_{ \mathbf{x} } \Big[ E_{ K( \tau_{\epsilon} ) }  \big[\chi_{M} \big] \Big] \geq 1 -\epsilon.
$$
Since $\epsilon$ was arbitrary, the theorem follows. 
\end{proof}

\bigskip
\bigskip
\begin{re}
Theorem 4 in \cite{GS72} is stated for a jump-diffusion with a compensated Poisson measure. However, since we assumed that $\nu \big( \mathbb{R} \big)<\infty$, we may rewrite Equation (1) as 
\begin{equation*}\begin{split}
dS(t) & = \bigg(-\beta S(t) I(t) -\mu S(t) + \mu + \lambda I(t) + \int_{ \mathbb{R} } h(y) S(t-) I(t-)\nu(dy) \bigg)dt -\sigma  S(t) I(t) dW(t)  + \int_{ \mathbb{R} } h(y) S(t-) I(t-) \tilde{N}(dt,dy), 
\\ dI(t) & = \bigg( \beta S(t) I(t) - (\lambda +\mu) I(t) -  \int_{ \mathbb{R} } h(y) S(t-) I(t-)\nu(dy) \bigg)dt + \sigma S(t) I(t) dW(t)  - \int_{ \mathbb{R} } h(y) S(t-) I(t-) \tilde{N}(dt,dy),
\end{split}\end{equation*}
and thus we may apply this theorem. The generator remains unchanged.
\end{re}

\bigskip
\bigskip
\begin{cor}
Suppose that $h(y)$ nonnegative, and there exists $0<\varphi<1$ such that $\beta < \mu + \lambda + \varphi \int_{ \mathbb{R} } h(y) \nu(dy)$. Then
$$
P_{ \mathbf{x}  } \bigg( \lim_{t \to \infty} K(t) = (1,0) \bigg)=1.
$$
\end{cor}
\begin{proof}
Taking a neighborhood $U= \big\{ \mathbf{x} \in \Delta_2 : x_1 > \varphi \big\}$ and $g(x)$ as above, we have that 
$$
L_0 g(x) \leq - \bigg( - \beta  + \mu + \lambda +  \varphi \int_{ \mathbb{R} } h(y)\nu(dy) \bigg)\big( 1-x \big).
$$
Theorem 4 in \cite{GS72} and Remark 2 gives us, for $\mathbf{x} \in U$,
$$
P_{ \mathbf{x}  } \bigg( \lim_{t \to \infty} K(t) = (1,0) \bigg) \geq 1 -\epsilon_0,
$$
for some $\epsilon_0$. The rest of the proof follows as above
\end{proof}

\bigskip
\bigskip
\begin{re}
The Remark 2 \cite{GS72} holds for continuous processes. However, since the jumps are positive, this only helps the convergence of the sample paths to $(1,0)$. Therefore, we are able to use the result.
\end{re}

\bigskip
\bigskip
The computer simulations below agree with the conclusion of the theorems. Figure 1 shows for both cases that the disease free equilibrium is globally asymptotically stable. Figure 2 tell us that the process is recurrent and thus simulates an epidemic. 

%Transience
\begin{figure}[h!]
\begin{center}
\subfigure[The negative integral case.]{ \includegraphics{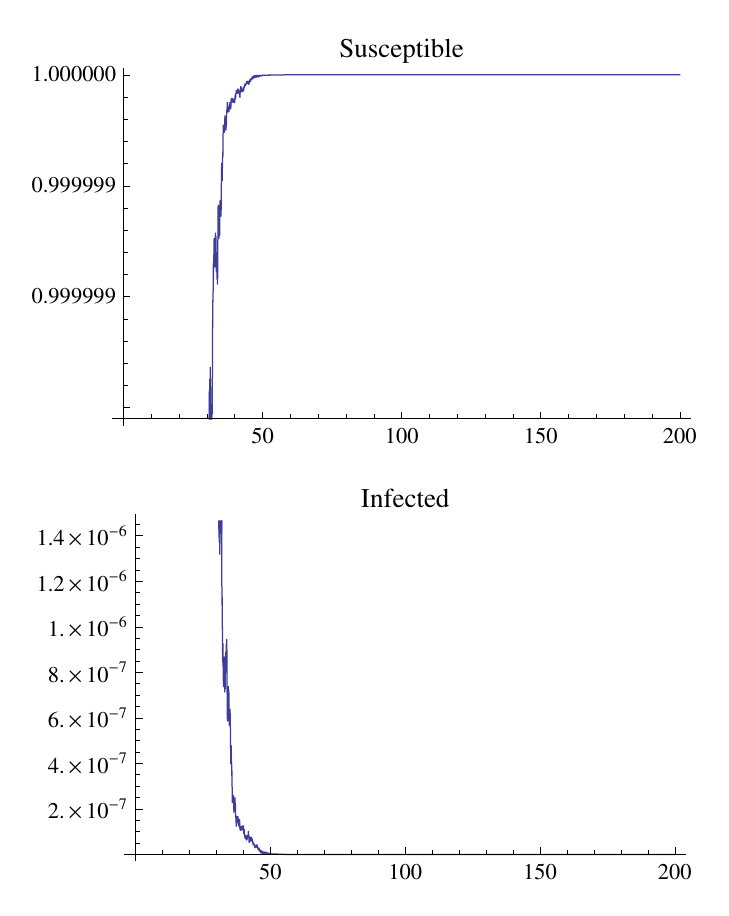} }
\subfigure[The positive jump function case.]{ \includegraphics{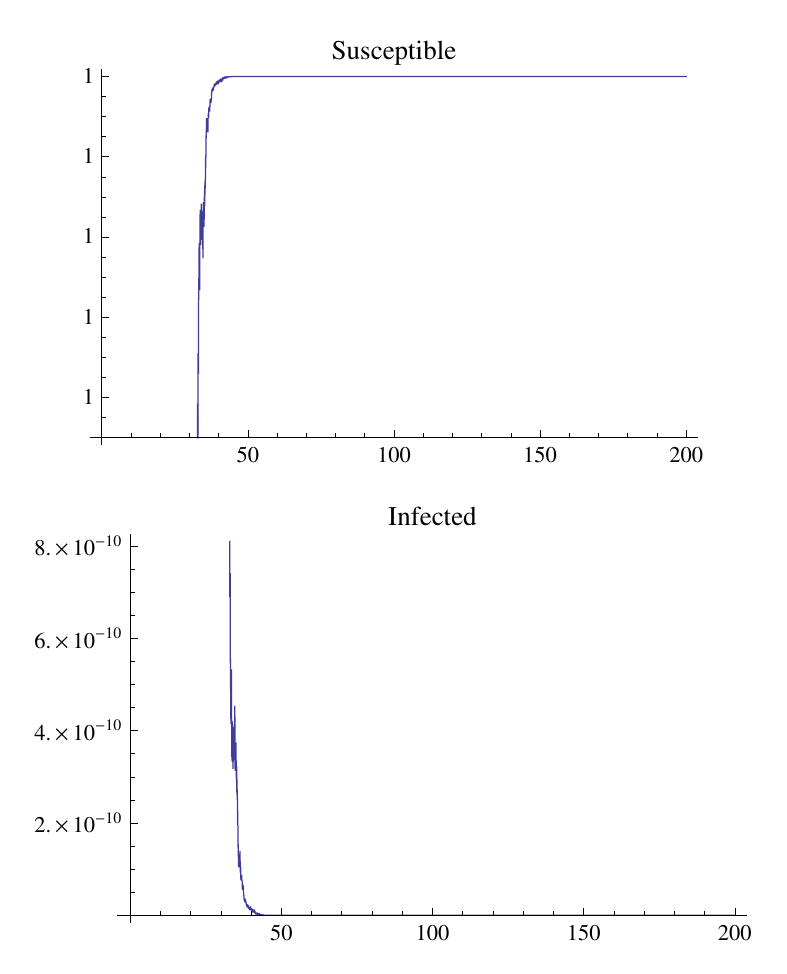} \label{fig:sec}  }
\end{center}
\caption{ For simulations, the initial condition are $S_0=.6$ and $I_0=.4$.  The parameter values for Figure 1(a) are $\beta=.1$, $\sigma=.3$, $\lambda=.3$,  $\mu=.2$, $\nu\big(\mathbb{R} \big)=1$,  and $j(y) \equiv -.01$, while the parameter values for Figure 1(b) are $\beta=.4$, $\sigma=.3$, $\lambda=.3$,  $\mu=.1$, $\nu\big(\mathbb{R} \big)=.5$,  and $j(y) \equiv .1$.}
\end{figure}

%Reccuence
\begin{figure}[h!]
\begin{center}
\subfigure[The negative integral case.]{ \includegraphics{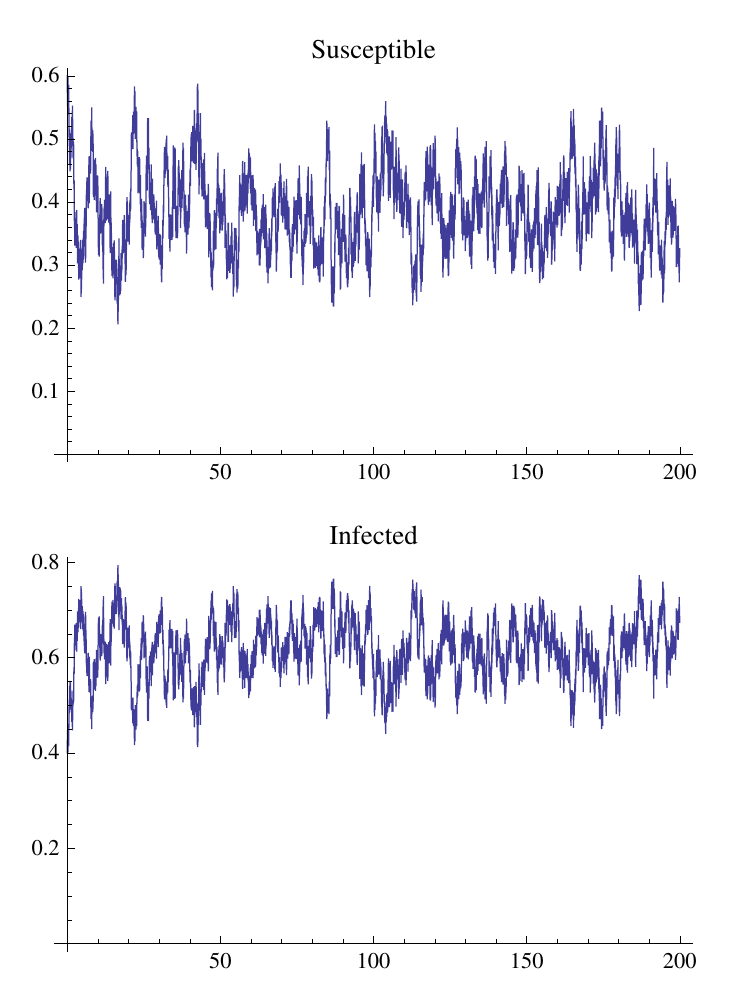} }
\subfigure[The positive jump function case.]{ \includegraphics{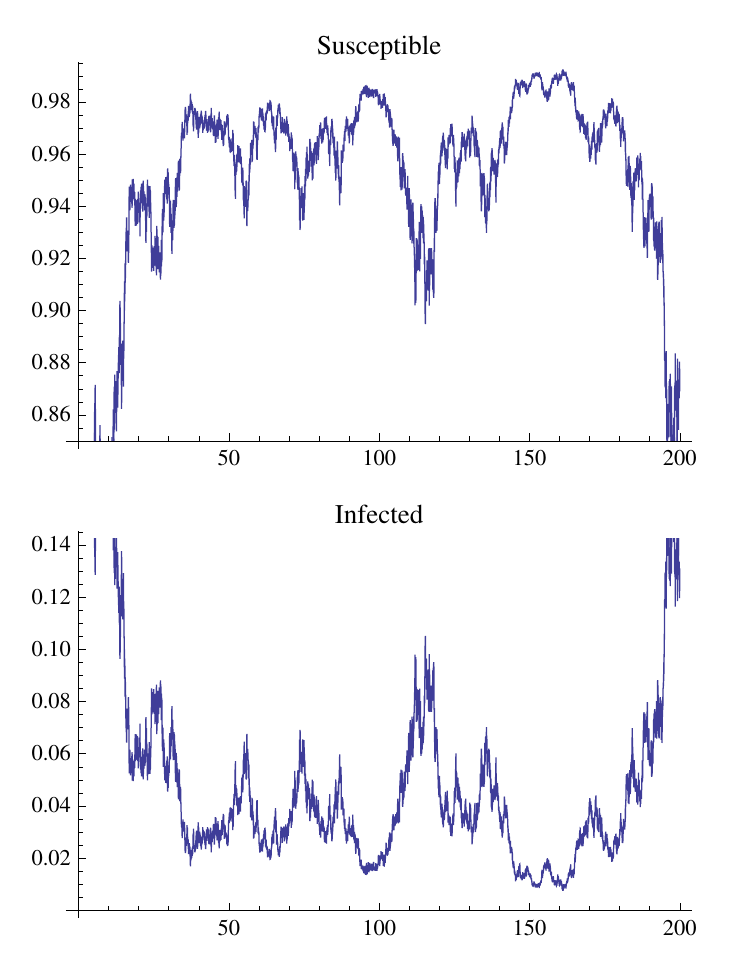} \label{fig:sec}  }
\end{center}
\caption{ Just as in Figure 1, the initial condition are $S_0=.6$ and $I_0=.4$. For Figure 2(a), the parameter values are $\beta=.4$, $\sigma=.3$, $\lambda=.3$,  $\mu=.15$, $\nu\big(\mathbb{R} \big)=1$,  and $j(y) \equiv -.1$, and for Figure 2(b), the parameter values are $\beta=.8$, $\sigma=.3$, $\lambda=.2$,  $\mu=.1$, $\nu\big(\mathbb{R} \big)=2$,  and $j(y) \equiv .1$.   }
\end{figure}

\bigskip
\bigskip
\begin{re}
If $\displaystyle \int_{ \mathbb{R} } h(y)\nu(dy)> 0$ and the function $j(y)$ was able to take both positive and negative values, we were not able to say anything about this case.  A way to calculate the stability of a system of this type is to determine if and where $S(t)$ leaves an arbitrary subinterval of $[0,1]$. So, for any $0 < x_1< x_2 <1$, define $\displaystyle \tau_{x_1x_2}(x_0)=\inf_{t \geq 0}\Big\{ S(t) \not\in(x_1,x_1) \big| S(0)=x_0 \Big\}$,  $\pi_{x_2;x_1}(x_0)=P\Big( S\big( \tau_{x_1x_2}(x_0) \big) \geq x_2 \Big)$, and $\pi_{x_1;x_2}(x_0)=P\Big( S\big( \tau_{x_1x_2}(x_0) \big) \leq x_1 \Big)$. The papers of \cite{HT76} and \cite{MA00} tell us that solving the integro-differential equation

$$
\alpha(x) u'(x) + \frac{\gamma^2(x)}{2} u''(x) + \int_{ \mathbb{R} } \bigg[ u\bigg( x + h(y) x (1-x) \bigg) - u(x) \bigg] \nu(dy) = 0,
$$
with the initial conditions of $u(x)=1$ for $x\in[x_2,1]$ and $u(x)=0$ for $x\in[0,x_1]$, will give us $\pi_{x_2;x_1}(x_0)$ (and accordingly for $\pi_{x_1;x_2}(x_0)$).  Taking $x_0 \to 0$, and $x_1\to 1$ will tell us how this system evolves. This is similar to the method given in Gihman and Skorohod \cite{GS72}.
\end{re}

%SIRS
\bigskip
\bigskip
\section{Analysis of the SIRS Model}
To analyze the SIRS model, we will use the stochastic Lyapunov method. For a bit of simplicity, we will define $\displaystyle \mathfrak{E}(t)= \big( S(t), I(t), R(t) \big)$.

\bigskip
\begin{no}
Looking closely at the process, one can see that the only stationary position is the point $e_1=(1,0,0)$. 
\end{no}

\bigskip
\bigskip

\begin{thm}
If $\displaystyle \beta < \min \bigg\{ \lambda +  \int_{ \mathbb{ R} } j(y)\nu(dy)  - \int_{ \mathbb{ R} } \frac{  j^2(y) }{ 2 } \nu(dy) -\frac{ \sigma^2 }{ 2}, \ \ \delta \bigg \}$ then 
$$
P_x \bigg( \lim_{ t \to \infty} \mathfrak{E}(t) = \mathbf{e}_1 \bigg) = 1,
$$
\end{thm}
where $\displaystyle x\in  \Delta_3$.
\begin{proof}

\bigskip
\bigskip
The proof will follow the one given for the SIS model. For $\mathfrak{L}$ the infinitesimal generator of the SIRS stochastic process, we see that 
\begin{equation*}\begin{split}
\mathfrak{L}g(x) & =\Big( -\beta x_1 x_2 + \delta x_3 \Big) \frac{\partial g}{\partial x_1}(x) + \Big( \beta x_1 x_2 - \lambda  x_2 \Big)\frac{\partial g}{\partial x_2}(x) + \Big(  \lambda x_2 - \delta x_3 \Big) \frac{\partial g}{\partial x_3}(x) + \frac{1}{2} \sigma^2 x^2_1x^2_2 \frac{\partial^2 g}{\partial x_1^2}(x)
\\ &  + \frac{1}{2} \sigma^2 x^2_1x^2_2 \frac{\partial^2 g}{\partial x_2^2}(x) - \frac{1}{2} \sigma^2 x^2_1x^2_2 \frac{\partial^2 g}{\partial x_1 \partial x_2}(x) - \frac{1}{2} \sigma^2 x^2_1x^2_2 \frac{\partial^2 g}{\partial x_2 \partial x_1 }(x) 
\\ & + \int_{ \mathbb{R} } \bigg[ g\Big( x+ \big(0,- j(y)x_2,  j(y)x_2\big) \Big) - g(x) \bigg]\nu(dy)
\end{split}\end{equation*}

Take $\kappa$ as a positive constant such that $\delta - \beta -2 \kappa>0$, (which exists by our assumption). Now, define the positive function $f(x)=c_1\big( x_1 -1 \big)^2 + c_2x_2^2 + c_3 x_3^2$, where $c_1$, $c_2$, and $c_3$ are positive, 
$$
\frac{ c_3 \bigg( \lambda +  \int_{ \mathbb{R} } j(y) \nu(dy) \bigg)  } { \delta - \beta -2 \kappa } < c_1,
$$ 
and 
$$
\frac{ c_1\Big( \frac{ \sigma^2}{2} + \beta \Big) + c_3 \int_{ \mathbb{R} } j^2(y) \nu(dy)  }{  \lambda +  \int_{ \mathbb{ R} }  j(y)\nu(dy)  - \int_{ \mathbb{ R} } \frac{  j^2(y) }{ 2 } \nu(dy)   -  \frac{ \sigma^2 }{ 2}  - \beta  } < c_2 .
$$ 
Then
\begin{equation*}\begin{split}
\mathfrak{L}f(x) & =\Big( -\beta x_1 x_2 + \delta x_3 \Big) 2c_3 \big( x_1 - 1 \big) + \Big( \beta x_1 x_2 - \lambda x_2 \Big)c_2 x_2 + \Big( \lambda x_2 - \delta x_3 \Big) c_3x_3 + \frac{1}{2} \sigma^2 x^2_1x^2_2 2c_1 +  \frac{1}{2}  \sigma^2 x^2_1x^2_2 2c_2
\\  & \ \ \ \  + \int_{ \mathbb{R} } \bigg[ c_2\Big( x_2 - j(y)x_2 \Big)^2 - c_2x_2^2 \bigg]\nu(dy) + \int_{ \mathbb{R} } \bigg[c_3 \Big( x_3 +  j(y)x_2 \Big) - c_3 x_3^2 \bigg]\nu(dy)
\\ & = \bigg\{ c_1\sigma^2 x_1^2 x_2^2 -2c_2   x_2^2 + c_2 \int_{ \mathbb{ R} } \Big[ -2  j(y) +  j^2(y) \Big] \nu(dy) x_2^2  + c_2 \sigma^2 x_1^2 x_2^2 +2 c_2 \beta x_1x_2^2 + c_3\int_{ \mathbb{ R} } j^2(y) \nu(dy) x_2^2  \bigg\}
\\  & + \bigg\{ -2c_1\delta x_3 + 2c_1 \delta x_1 x_3 + c_3 \lambda x_2x_3 + 2c_3\int_{ \mathbb{ R} } j(y) \nu(dy)x_2 x_3 -2c_3 \delta x_3^2  \bigg\} + \bigg\{ -2c_1 \beta x_2 x_1^2 + 2c_1\beta x_2 x_1 \bigg\}.
\end{split}\end{equation*}
Organizing and simplifying yields
\begin{equation*}\begin{split}
\mathfrak{L}f(x) & = 2\bigg\{ c_1 \frac{ \sigma^2 }{ 2 } x_1^2  - c_2 \bigg( \lambda +  \int_{ \mathbb{ R} }  j(y)\nu(dy)  - \int_{ \mathbb{ R} } \frac{  j^2(y) }{ 2 } \nu(dy)   -  \frac{ \sigma^2 }{ 2}  x_1^2 - \beta x_1 \bigg)+ c_3\int_{ \mathbb{ R} } j^2(y) \nu(dy)\bigg\} x_2^2  
\\  & + 2\bigg\{ c_1\delta\Big(-1 + x_1 \Big)  + c_3 \bigg( \lambda x_2+ \int_{ \mathbb{ R} } j(y) \nu(dy)x_2  - \delta x_3 \bigg)  \bigg\} x_3 + 2c_1\beta \bigg\{ - x_1 + 1  \bigg\} x_2 x_1
\\ & = 2\bigg\{ c_1 \frac{ \sigma^2 }{ 2 } x_1^2  - c_2 \bigg( \lambda +  \int_{ \mathbb{ R} } j(y)\nu(dy)  - \int_{ \mathbb{ R} } \frac{  j^2(y) }{ 2 } \nu(dy)   -  \frac{ \sigma^2 }{ 2}  x_1^2 - \beta x_1 \bigg)+ c_3\int_{ \mathbb{ R} } j^2(y) \nu(dy)\bigg\} x_2^2  
\\  & + 2\bigg\{ c_1\delta\Big(-x_2 - x_3 \Big)  + c_3 \bigg( \lambda x_2+ \int_{ \mathbb{ R} }  j(y) \nu(dy)x_2  - \delta x_3 \bigg)  \bigg\} x_3 + 2c_1\beta \bigg\{ x_2 + x_3 \bigg\} x_2 x_1
\\ & = 2\bigg\{ c_1 \bigg( \frac{ \sigma^2 }{ 2 } x_1^2 + \beta x_1\Bigg) - c_2 \bigg( \lambda +  \int_{ \mathbb{ R} }  j(y)\nu(dy)  - \int_{ \mathbb{ R} } \frac{ j^2(y) }{ 2 } \nu(dy)   -  \frac{ \sigma^2 }{ 2}  x_1^2 - \beta x_1 \bigg)+ c_3\int_{ \mathbb{ R} } j^2(y) \nu(dy)\bigg\} x_2^2  
\\  & + 2\bigg\{ \bigg( -c_1 \Big( \delta - \beta x_1x_2 \Big)  + c_3 \bigg[ \lambda +  \int_{ \mathbb{ R} } j(y) \nu(dy) \bigg]  \bigg) x_2  - (c_1+c_3 ) \delta x_3 \bigg)  \bigg\} x_3.
\end{split}\end{equation*}
Increasing the values of the positive numbers, we see that 
\begin{equation*}\begin{split}
\mathfrak{L}f(x) & \leq 2\bigg\{ c_1 \bigg( \frac{ \sigma^2 }{ 2 } + \beta \Bigg) - c_2 \bigg( \lambda +  \int_{ \mathbb{ R} }  j(y)\nu(dy)  - \int_{ \mathbb{ R} } \frac{ j^2(y) }{ 2 } \nu(dy)   -  \frac{ \sigma^2 }{ 2}  - \beta  \bigg)+ c_3 \int_{ \mathbb{ R} } j^2(y) \nu(dy)  \bigg\} x_2^2  
\\  & + 2\bigg\{ \bigg( -c_1 \Big( \delta - \beta \Big)  + c_3 \bigg[ \lambda +  \int_{ \mathbb{ R} }  j(y) \nu(dy) \bigg]  \bigg) x_2  - (c_1+c_3 ) \delta x_3 \bigg)  \bigg\} x_3 
\\ & = 2\bigg\{ c_1 \bigg( \frac{ \sigma^2 }{ 2 } + \beta \Bigg) - c_2 \bigg( \lambda +  \int_{ \mathbb{ R} }  j(y)\nu(dy)  - \int_{ \mathbb{ R} } \frac{ j^2(y) }{ 2 } \nu(dy)   -  \frac{ \sigma^2 }{ 2}  - \beta  \bigg)+ c_3 \int_{ \mathbb{ R} } j^2(y) \nu(dy)  \bigg\} x_2^2  
\\  & + 2\bigg\{ \bigg( -c_1 \Big( \delta - \beta \Big)  + c_3 \bigg[ \lambda +  \int_{ \mathbb{ R} }  j(y) \nu(dy) \bigg]  \bigg) x_2  - \big(c_1\delta +c_3 \delta \big)  x_3 \bigg)  \bigg\} x_3  -\kappa c_1(x-1)^2 +  \kappa c_1(x-1)^2
\\ & = 2\bigg\{ c_1 \bigg( \frac{ \sigma^2 }{ 2 } + \beta + \kappa \Bigg) - c_2 \bigg( \lambda +  \int_{ \mathbb{ R} }  j(y)\nu(dy)  - \int_{ \mathbb{ R} } \frac{ j^2(y) }{ 2 } \nu(dy)   -  \frac{ \sigma^2 }{ 2}  - \beta  \bigg)+ c_3 \int_{ \mathbb{ R} } j^2(y) \nu(dy)  \bigg\} x_2^2  
\\  & + 2\bigg\{ \bigg( -c_1 \Big( \delta - \beta - 2\kappa \Big)  + c_3 \bigg[ \lambda +  \int_{ \mathbb{ R} }  j(y) \nu(dy) \bigg]  \bigg) x_2  -  \Big(c_1  \big( \delta -\kappa  \big) +c_3\delta  \Big)x_3 \bigg)  \bigg\} x_3  -\kappa c_1(x-1)^2 
\\ & \leq 2\bigg\{ c_1 \bigg( \frac{ \sigma^2 }{ 2 } + \beta + \kappa \Bigg) - c_2 \bigg( \lambda +  \int_{ \mathbb{ R} }  j(y)\nu(dy)  - \int_{ \mathbb{ R} } \frac{ j^2(y) }{ 2 } \nu(dy)   -  \frac{ \sigma^2 }{ 2}  - \beta  \bigg)+ c_3 \int_{ \mathbb{ R} } j^2(y) \nu(dy)  \bigg\} x_2^2  
\\  & - 2\Big(c_1  \big( \delta -\kappa  \big) +c_3\delta  \Big) x^2_3  -\kappa c_1(x-1)^2. 
\end{split}\end{equation*}
Therefore, there exists a positive constant $k$, such that $\displaystyle \mathfrak{L}g \leq -k g$. Theorem 4 in \cite{GS72} tells us that for an $\epsilon>0$, there exists a neighborhood of $\mathbf{e}_1$, say $U$, such that 
$$
P_{ \mathbf{x}  } \bigg( \lim_{t \to \infty} \mathfrak{E}(t) = \mathbf{e}_1  \bigg) \geq 1 -\epsilon.
$$
for $\mathbf{x} \in U \cap \Delta_2$.

\medskip
Define $\tau_{\epsilon}$ as above, and fix an arbitrary $\epsilon>0$. We will show that $\displaystyle E_{ \mathbf{x} } \Big[ \tau_{\epsilon}  \Big]<\infty$, which will assert our theorem. To adjust for the process possibly jumping over the boundary of $\tau_{\epsilon}$, we will define $1-\epsilon_0=\sup\Big\{ S\big(\tau_{\epsilon} \big) \Big\}$, (where $\epsilon_0=1$ if the set is empty). Since the process does not hit the boundary in finite time, we have that $\epsilon_0>0$ a.s. Note that, for the function $f$ above, $\displaystyle \mathfrak{L}f(x) \leq \alpha \big( 1- x \big)^2$, for some $\alpha>0$. Now for $\mathbf{x} \in \Delta_3$, such that $x_1 < 1- \epsilon$, and $0<t<\infty$, Dynkin's formula yields
\begin{equation*}\begin{split}
E_{ \mathbf{x} } \Big[ f\big( \tau_{\epsilon} \wedge t \big) \Big] & = f(x) - E_{ \mathbf{x} } \bigg[ \int_{0}^{  \tau_{\epsilon} \wedge t  } \mathfrak{L}f\Big( \mathfrak{E}(t)\Big) ds \bigg]
\\ & \leq  f(x) - E_{ \mathbf{x} } \bigg[ \int_{0}^{  \tau_{\epsilon} \wedge t  } \Big( 1- S(t)\Big)^2 ds \bigg]
\\ & \leq  f(x) - \alpha \epsilon^2_0 E_{ \mathbf{x} } \Big[ \tau_{\epsilon} \wedge t  \Big].
\end{split}\end{equation*}
Thus $\displaystyle E_{ \mathbf{x} } \Big[ \tau_{\epsilon}  \Big]<\infty$, and therefore, following the proof of Theorem 1,
$$
P_x \bigg( \lim_{ t \to \infty} \mathfrak{E}(t) = \mathbf{e}_1 \bigg) = 1.
$$

\end{proof}

\bigskip
\bigskip
The simulations given below show globably asymptotic stability to the disease free equilibrium and an epidemic.

\newpage

\begin{figure}[!] 
\centering
\subfigure[Convergence to the disease free equilibrium.]{ \includegraphics{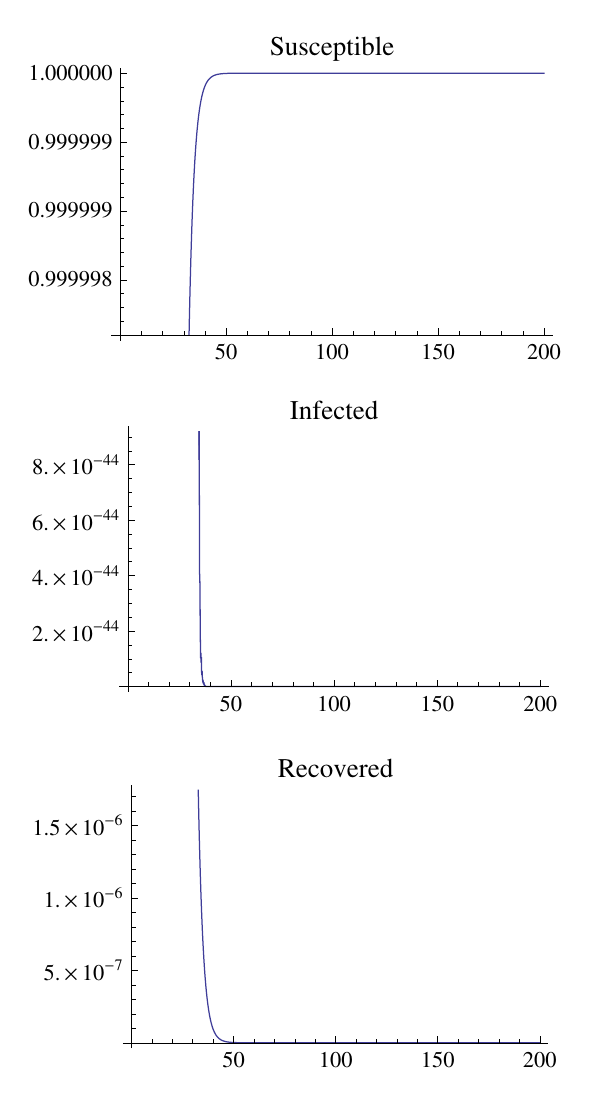} }
\subfigure[An example of an epidemic.]{ \includegraphics{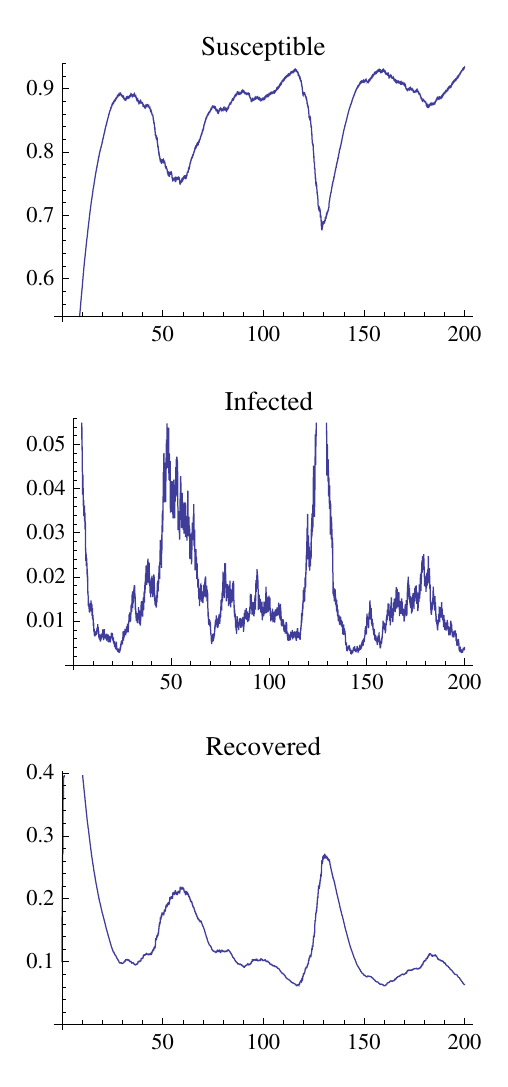} }
\caption{ For both simulations, the initial condition are $S_0=.3$, $I_0=.6$, and $R_0=.1$.  For the left hand side the parameter values $\beta=.3$, $\sigma=.1$, $\lambda=.29$,  $\delta=.4$, $\nu\big(\mathbb{R} \big)=1$,  and $j(y) \equiv .3$ were used. While the parameter values of $\beta=.8$, $\sigma=.2$, $\lambda=.1$,  $\delta=.1$, $\nu\big(\mathbb{R} \big)=.5$,  and $j(y) \equiv .1$ were used. }
\end{figure}

\begin{acknowledgement*}
The author would like to thank Professor Troy Day for wonderful guidance.
\end{acknowledgement*}

\bigskip
\bibliography{epi}                                                                                                                                               
\bibliographystyle{plain}                                                                                                                                                                              
\nocite{*}

\end{document}